\documentclass[12pt, reqno]{amsart}
\allowdisplaybreaks[1]
\usepackage{amsmath}
\usepackage{amssymb}
\usepackage{amsfonts}
\usepackage{verbatim}
\usepackage[usenames]{color}
\usepackage{hyperref}
\makeindex

 \newtheorem{theorem}{Theorem}[section]
 \newtheorem{corollary}[theorem]{Corollary}
 \newtheorem{lemma}[theorem]{Lemma}
 \newtheorem{proposition}[theorem]{Proposition}

\theoremstyle{definition}
\newtheorem{definition}[theorem]{Definition}
\theoremstyle{remark}

\newtheorem{fact*}{Fact}

\DeclareMathOperator{\RE}{Re}

\newcommand{\hilbert}{\mathcal{H}}

\newcommand{\T}{\mathbb{T}}

\newcommand{\D}{\mathbb{D}}
\newcommand{\C}{\mathbb{C}}

\newcommand{\BH}{\mathrm{B}(\mathcal{H})}

\newcommand{\norm}[1]{\left\Vert#1\right\Vert}

\newcommand{\ip}[2]{\left\langle #1, #2 \right\rangle}
\newcommand{\ad}{^\ast}
\newcommand{\inv}{^{-1}}

\newcommand{\til}{\raise.17ex\hbox{$\scriptstyle\mathtt{\sim}$}}

\newcommand\la{\lambda}

\newcommand\beq{\begin{equation}}

\newcommand\eeq{\end{equation}}

\newcommand{\vectwo}[2]
{
   \begin{pmatrix} #1 \\ #2 \end{pmatrix}
}

\newcommand\blue{\color{black}}
\newcommand\black{\color{black}}
\newcommand\red{\color{black}}
\newcommand\notred{\color{black}}
\newcommand{\bbm}{\left[ \begin{smallmatrix}}
\newcommand{\ebm}{\end{smallmatrix} \right]}
\newcommand{\bpm}{\left( \begin{smallmatrix}}
\newcommand{\epm}{\end{smallmatrix} \right)}
\numberwithin{equation}{section}

\newcommand{\tensor}[2]{\text{ }{\begin{smallmatrix} #1 \\ \otimes\\ #2\end{smallmatrix}}\text{  }}

\newlength{\Mheight}
\newlength{\cwidth}
\newcommand{\mc}{\settoheight{\Mheight}{M}\settowidth{\cwidth}{c}M\parbox[b][\Mheight][t]{\cwidth}{c}}

\newcommand{\dfn}[1]{{\bf #1}\index{#1}}
\newcommand{\BallB}[1]{\mathrm{Ball}(#1)}

\newcommand{\RHPB}[1]{\mathrm{RHP}(#1)}
\newcommand{\MU}[1]{\mathcal{M}(#1)}
\newcommand{\regherglotzclass}[2]{\mathrm{RHerglotz}(#1,#2)}
\newcommand{\repregherglotzclass}[2]{\mathrm{RRHerglotz}(#1,#2)}
\newcommand{\Free}[2]{\mathrm{Free}(#1, #2)}

\title[Representation of free Herglotz functions]{Representation of free Herglotz functions}
\author{
J. E. Pascoe$^\dagger$
}
\address{Department of Mathematics\\
  Washington University in St. Louis\\
  One Brookings Drive \\
 St. Louis, MO 63130}
\email[J. E. Pascoe]{pascoej@math.wustl.edu}

\thanks{$\dagger$ Partially supported by National Science Foundation Mathematical
Science Postdoctoral Research Fellowship  
DMS 1606260}
\author{
Benjamin Passer$^\ddagger$
}
\address{Mathematics Department \\
Technion - Israel Institute of Technology \\
Haifa, 320003\\
Israel}
\email[B. Passer]{\blue benjaminpas@technion.ac.il \black}
\thanks{\blue $\ddagger$ Partially supported by a Zuckerman fellowship at the Technion. \black}
\author{
Ryan Tully-Doyle
}
\address{Department of Mathematics and Physics \\
University of New Haven\\
West Haven, CT 06516 }
\email[R. Tully-Doyle]{rtullydoyle@newhaven.edu}

\date{\today}


\begin{document}

\begin{abstract}
A Herglotz function is a holomorphic map from the open complex unit disk into the closed complex right halfplane. A classical Herglotz function has an integral representation against a positive measure on the unit circle. We prove a free analytic analogue of the Herglotz representation and describe how our representations specialize to the free probabilistic case. We also show that the set of representable Herglotz functions arising from noncommutative conditional expectations must be closed in a natural topology.
\end{abstract}
\maketitle

\tableofcontents

\section{Introduction}
A classical Herglotz function is a holomorphic map from the unit disk into the complex right half plane. The following representation characterizes a Herglotz function as an integral against a probability measure on the unit circle $\T$.
\begin{theorem}[Herglotz \cite{herg11}] \label{herglotzclassic}
A holomorphic function $h$ defined on $\D$ satisfies $\RE h \geq 0$ and $h(0) = 1$ if and only if there exists a (unique) probability measure $\mu$ supported on $\T$ such that 
	\beq\label{classherg}h(x) = \int_\T \frac{1 + e^{i\theta}x}{1 - e^{i\theta}x}d\mu(e^{i\theta}).\eeq
\end{theorem}

\red The correct analogue of the Herglotz representation in several complex variables is somewhat elusive. However, in two complex variables, a useful theory of Herglotz representations was developed by J. Agler in \cite{ag90}, along the lines of operator theory. For functions in one variable, we see that the Herglotz representation given in Equation \eqref{classherg} can be reinterpreted as
\beq\label{opherg}
 h(x) = \ip{(1 + Ux)(1 - Ux)\inv \alpha}{\alpha},
\eeq
where $U$ is multiplication by $z$ acting on $L^2(\mu)$ and $\alpha$ is the constant function $1 \in L^2(\mu)$. It can be shown that for an arbitary unitary $U$ and unit vector $\alpha$, the formula \eqref{opherg} defines a function taking the disk to the right half-plane.

In \cite{ag90}, Agler extends \eqref{opherg} to two variables: any analytic function $h$ from $\D^2$ satisfying $\RE h \geq 0$ with $h(0) = 1$ must be of the form
\[
 h(x) = \ip{(1 + Ux_P)(1 - Ux_P)\inv \alpha}{\alpha}
\]
where $U$ is a unitary, $\alpha$ is a unit vector, and $x_P = x_1 P + x_2 (1 - P)$ for a projection $P$.

A further reinterpretation of this along current lines of inquiry in free probability (e.g. \cite{will13, speich} among many others) would be
\[
 h(x) = R((1 + Ux)(1 - Ux)\inv)
\]
where $U$ is a unitary contained in some $C^\ast$-algebra $B$ containing $\C^2$ as a unitally included subalgebra, and $R$ is a (completely) positive map from $B$ to $\C$ . Here, $B$ is the algebra generated by $U$ and $P$ in $B(\mathcal H)$, the inclusion of $\C^2$ is given by $(x_1, x_2) \mapsto x_1 P + x_2 (1-P)$, and $R$ is the (completely) positive map $R(b) = \ip{b\alpha}{\alpha}$. 
\black

Herglotz functions have been generalized to many contexts. Recent work by Michael Anshelevich and John D. Williams has dealt with a related class of functions from a noncommutative upper half plane into itself and various analogues of ``integral representations'' arising in free probability \cite{anw14, will13}. However, an exact correspondence for free probability as in Theorem \ref{herglotzclassic} is unknown. 
In Theorem \ref{exacthasrep}, we show that the correspondence exists.
However, the analogue of a measure involved in the representation is highly non-unique. \red Free probability itself has found may applications including random matrix theory (see \cite{speich} for an introduction), and we hope that the theory of Herglotz representations and their relatives will allow ideas from classical probability, functional analysis, and moment theory to be incorporated into the framework of free probability and its applications. \black

\subsection{The noncommutative context}
Let $B$ be a unital $C^*$-algebra. 
The \dfn{matrix universe over $B$}, denoted $\MU{B},$ is the set of square matrices over $B$, \blue which we write as the disjoint union \black
	$$\MU{B} = \bigcup^{\infty}_{n=1} M_n(B).$$
\blue That is, $\MU{B}$ is graded over the set of positive integers. \black
Next, the \dfn{ball over $B$,} denoted $\BallB{B},$ is the set of \blue strictly \black contractive matrices over $B$:
	$$\BallB{B} =  \{X \in \MU{B} | \hspace{2pt} \|X\|<1\}.$$
Similarly, the \dfn{right half plane over $B$,} denoted $\RHPB{B},$ is
	$$\RHPB{B} =  \{X \in \MU{B} | \hspace{2pt} \text{Re } X \geq 0\},$$
{\notred where $\text{Re } X = (X + X^*)/2$}.
For any $\mathcal{D} \subset \MU{B_1},$ a \dfn{free function} 
$f: \mathcal{D} \rightarrow \MU{B_2}$ is \blue a function which is \black graded and respects intertwining maps. That is,
\blue
\begin{enumerate}
\item $f(M_n(B_1) \cap \mathcal{D}) \subseteq M_n(B_2)$, and
\item if $\Gamma X = Y\Gamma$ for $X, Y \in \mathcal{D}$ and a rectangular matrix $\Gamma$ of scalars, then $\Gamma f(X) = f(Y) \Gamma$.
\end{enumerate}
In the above, we identify a scalar $z \in \mathbb{C}$ with $z \cdot I \in B$, so we may consider a rectangular matrix $\Gamma$ of scalars as a matrix over $B$. Note in particular that $\Gamma$ need not be square. \black
We denote the set of free functions \red from \black $\mathcal{D}$ to $\mathcal{R}$
by  $\Free{\mathcal{D}}{\mathcal{R}}.$

In this noncommutative context, a \dfn{free Herglotz function} is just a free function $h: \BallB{B_1} \rightarrow \RHPB{B_2}.$ We call a Herglotz function \dfn{regular} if $h(0) =I$, and $h$ has a \dfn{regular Herglotz representation} if there exists
\begin{enumerate}
	\item a $C^*$-algebra $M$ unitally containing $B_1,$
	\item a completely positive unital linear map $R: M \rightarrow B_2$, and 
	\item a unitary $U \in M,$
\end{enumerate}
such that
\beq \label{herglotzintro} h(X) =\tensor{R}{1_n}\left[\left(I + \tensor{U}{I_n} X \right)\left(I - \tensor{U}{I_n} X \right)^{-1}\right] .\eeq
We have adopted a \emph{vertical tensor notation} to save space: $\tensor{A}{B}$
represents the same object as $A \otimes B$.
Here, $1_n$ represents the identity map on $n \times n$
matrices, and $I_n$ represents the $n \times n$ identity
matrix.
We denote the \dfn{set of all regular Herglotz functions} $h: \BallB{B_1} \rightarrow \RHPB{B_2}$
by $\regherglotzclass{B_1}{B_2}$ and endow it with the topology of pointwise weak convergence. \red That is, a net  of functions $(f_\la)_\Lambda$ converges to a function $f$ if \blue and only if \red for every $X \in \BallB{B_1}$, $f_\la(X)$ converges weakly to $f(X)$. \blue This means that \red for every continuous linear functional $L$ on $B_2$, \blue the limit of \red $L(f_\la(X))$ \blue is \red $L(f(X))$. \black
We also denote the \dfn{set of all representable regular Herglotz functions} by $\repregherglotzclass{B_1}{B_2}.$

\subsection{Main results}
We prove two main results.
\begin{theorem} \label{exacthasrep}

$$\repregherglotzclass{B_1}{B_2} = \regherglotzclass{B_1}{B_2}.$$
That is, all regular Herglotz functions are representable.
\end{theorem}

Theorem \ref{exacthasrep} is proven below as Theorem \ref{blockherg}.
Via a standard calculation, we obtain the Herglotz representation when perhaps $h(0) \neq I$
in Corollary \ref{finalcor}. 
 The proof of Theorem \ref{exacthasrep} relies heavily on the Agler model theory which was developed in great generality by Ball, Marx and Vinnikov in the recent preprint \cite[Corollary 3.2]{BMV16}.

We note that the case of Theorem \ref{exacthasrep} where $B_1 = \mathbb{C}^n$ and $B_2 = \mathbb{C}$ was proven by Gelu Popescu \cite{popescumajorant}. An analogue of Theorem \ref{exacthasrep} for functions on the noncommutative upper half plane was shown by John Williams in \cite[Corollary 3.3]{will13} with different hypotheses; the function must analytically continue through some large set on the boundary, and asymptotic conditions are assumed to ensure that $R$ is given by a noncommutative conditional expection. We give 
conditions for a Herglotz function to arise from a Herglotz representation where
$R$ is given by a noncommutative conditional expectation in Theorem \ref{condcond}.

\begin{theorem} \label{introprop:condclosure}
The subset of $\repregherglotzclass{B}{B}$
such that $R$ can be chosen to be a noncommutative conditional expectation is closed in $\Free{\BallB{B}}{\RHPB{B}}$, where the latter is given the pointwise weak topology.
Moreover, if $B$ is a von Neumann algebra, then
the subset of 
	$\repregherglotzclass{B}{B}$ such that $R$ can be chosen to be a noncommutative conditional expectation  is compact.
\end{theorem}
Theorem \ref{introprop:condclosure} follows from Proposition \ref{prop:condclosure}. In general, if $B_2 \hookrightarrow C$ is a unital inclusion of $C^*$-algebras, then the pointwise weak topology of $\Free{\BallB{B_1}}{\RHPB{B_2}}$ is exactly the subspace topology of $\Free{\BallB{B_1}}{\RHPB{B_2}}$ inside $\Free{\BallB{B_1}}{\RHPB{C}}$ with the pointwise weak topology. Then one may consider the case where $C$ is the enveloping von Neumann algebra of $B_2$. 

Our Theorem \ref{introprop:condclosure} roughly corresponds to the work of Williams \cite{2015williams}, which gave that a certain set of noncommutative Cauchy transforms over a tracial von Neumann algebra should be closed in some sense.
The fact that Herglotz representations arise from unitaries as opposed to unbounded self-adjoint operators makes the theoretical concerns slightly less technical: we essentially show that the set of Herglotz representable functions is the continuous image of a compact space.  

Additionally, we show, in contrast to the classical case of Theorem \ref{herglotzclassic},
that a noncommutative Herglotz representation is highly non-unique in Propositions \ref{noninjectivity} and \ref{prop:condclosure}. For a general $C^*$-algebra not equal to $\mathbb{C}$, there exist Herglotz functions with non-unique representations, and Herglotz representations given by noncommutative conditional expectations can also be non-unique.


\section{The set of regular representable Herglotz functions as a topological space}
In the introduction, Herglotz functions were defined in terms of abstract $C^*$-algebras. We now will now give a somewhat more detailed description of the situation for concrete $C^*$-algebras, and prove our main results. \blue A free function $f: \mathrm{Ball}(B) \to \mathrm{RHP}(B(\mathcal{H}))$ with $f(0) = I$ will be called a $\textbf{(concrete) regular free Herglotz function}$.

Considering \eqref{classherg} and \eqref{opherg}, we are interested in when a concrete regular free Herglotz function $f: \mathrm{Ball}(B) \to \mathrm{RHP}(B(\mathcal{H}))$ admits a representation \black
\begin{equation}\label{eq:herglotzrep}
f(X) = \tensor{R}{1_n}\left[\left(I + \tensor{U}{I_n} \cdot \tensor{\alpha}{1_n}(X)\right)\left(I - \tensor{U}{I_n} \cdot \tensor{\alpha}{1_n}(X)\right)^{-1}\right]
\end{equation}
for $X\in \textrm{Ball}(B) \cap (B \otimes M_n(\mathbb{C}))$. Here $\widehat{\mathcal{H}}$ is a Hilbert space that contains $\mathcal{H}$, $U \in \mathcal{U}(\widehat{\mathcal{H}})$, $R: B(\widehat{\mathcal{H}}) \to B(\mathcal{H})$ is the natural restriction map, $\alpha: B \to B(\widehat{\mathcal{H}})$ is a unital representation, and $1_n: M_n(\mathbb{C}) \to M_n(\mathbb{C})$ is the identity map. The free Herglotz functions which may be written as in \eqref{eq:herglotzrep} are called \textbf{representable}, and one way to produce such functions passes through a universal construction. 
\blue
We note that the representation $\alpha$ used in \eqref{eq:herglotzrep} might not be faithful. However, if $\alpha$ is not faithful, there exists a faithful representation $\tilde{\alpha}$ which produces the same Herglotz function $f$. In particular, let $\beta: B \to B(\mathcal K)$ be a faithful representation of $B$ on some Hilbert space $\mathcal K$, and write $\tilde{\alpha} = \alpha \oplus \beta,$ $\tilde{U} = U \oplus I_{\mathcal K}$, and $\tilde{R}(x) = R(P_{\widehat{\mathcal{H}}}x|_{\hat{\hilbert}})$ for $x \in B(\widehat{\mathcal{H}} \oplus \mathcal{K})$. It follows that $\tilde{\alpha}$ is faithful and
\begin{equation}\label{eq:FAITHFULADDEDherglotzrep}
f(X) = \tensor{\tilde{R}}{1_n}\left[\left(I + \tensor{\tilde{U}}{I_n} \cdot \tensor{\tilde{\alpha}}{1_n}(X)\right)\left(I - \tensor{\tilde{U}}{I_n} \cdot \tensor{\tilde{\alpha}}{1_n}(X)\right)^{-1}\right]
\end{equation}   
for $X\in \textrm{Ball}(B) \cap (B \otimes M_n(\mathbb{C}))$.
\black This brings the concrete definition of a Herglotz representation into harmony with the original \blue definition\black.

  \subsection{Operator-valued states and their associated regular Herglotz representations}

\begin{definition}
Suppose $\mathcal{A}$ is a normed unital $*$-algebra, not necessarily complete. Let $\mathcal{H}$ be a fixed Hilbert space. A linear map $\phi: \mathcal{A} \to B(\mathcal{H})$ which satisfies
\begin{description}
\item[Unitality] $$\phi(1) = 1, \hfill$$
\item[Complete positivity] For any $a_1, \ldots a_n \in \mathcal A$, $$[\phi(a_j^*a_i)]_{1 \leq i,j\leq n} \geq 0,$$
\item[Complete boundedness] For any $a_1, \ldots, a_n, b \in \mathcal A$, $$[\phi(a_j^*(||b^*b|| - b^*b)a_i)]_{1 \leq i,j\leq n} \geq 0$$
\end{description}
is called an \textbf{operator-valued state} on $\mathcal{A}$. The collection of such $\phi$ is compact in the pointwise weak topology and is denoted $\Phi_\mathcal{A}^\mathcal{H}$. \red To see compactness explicitly, note \blue that for any $M \geq 0$, the set $B_M$ of elements in $B(\mathcal H)$ with norm at most $M$ is weakly compact, \red and one can show by elementary arguments that $\norm{\phi(a)} \leq \norm{a\ad a}^{1/2}$. \blue Therefore, we may embed $\Phi_\mathcal{A}^\mathcal{H}$ into $\Pi = \prod\limits_{a \in \mathcal{A}} B_{||a^*a||^{1/2}}$, where each $B_{||a^*a||^{1/2}}$ is given the weak topology, and $\Pi$ is given the product topology. By Tychonoff's theorem, $\Pi$ is compact. An examination of the definition shows that $\Phi_\mathcal{A}^\mathcal{H}$ is closed in $\Pi$, as the set of positive matrices over $B(\mathcal{H})$ is weakly closed, so $\Phi_\mathcal{A}^\mathcal{H}$ is compact. \black Moreover, when $\mathcal{A}$ is a $C^*$-algebra, \blue any positive element of $\mathcal{A}$, such as $||b^*b|| - b^*b$, may be written as $c^*c$ for some $c \in \mathcal{A}$. In this case, the third condition follows automatically from the second. \black
\end{definition}

If $B$ is a unital $C^*$-algebra, then let $\mathcal{A}_B$ denote the formal $*$-algebra containing $B$ and an additional unitary element $u$, with no other relations. \red That is, $\mathcal{A}_B$ is the {\notred free product of $B$ and $\mathbb{C}[\mathbb Z]$ in the category of unital complex $*$-algebras. Following the embedding of $\mathbb{C}[\mathbb{Z}]$ into $C^*(\mathbb{Z})$ shows that $\mathcal{A}_B$ embeds into the unital complex $*$-algebraic free product of $B$ and $C^*(\mathbb{Z})$. Finally, the proposition on page 429 of \cite{avit} gives that the latter object may be equipped with a pre-$C^*$ norm whose completion is the unital $C^*$-algebraic free product of $B$ and $C^*(\mathbb{Z})$. Since $\mathcal{A}_B$ then embeds densely into each object discussed, we have that if $\mathcal{R}_B$ is the collection of words in $B$ which represent $0$, then 
\begin{equation*}
\overline{\mathcal{A}_B} \cong C^*(B \cup \{u\} \hspace{2pt} | \hspace{2 pt} \mathcal{R}_B \cup \{u u^* - 1, u^*u - 1\}).
\end{equation*} 
In particular, no information in $\mathcal{A}_B$ is lost by restricting attention to representations on Hilbert spaces. }

If $\phi$ is an operator-valued state on $\mathcal{A}_B$ (which will frequently be obtained via restriction from $\overline{\mathcal{A}_B}$), then by mimicking the GNS construction in this operator setting, we may endow $\mathcal{A}_B \otimes_{\mathrm{alg}} \mathcal H$ with a sesquilinear form defined by 
\begin{equation*} \langle a_1 \otimes h_1, a_2 \otimes h_2 \rangle = \langle \phi(a_2^* a_1) h_1, h_2 \rangle_{\mathcal{H}}. \end{equation*}
While this form may be degenerate, a quotient and completion produce a Hilbert space $\mathcal{H}_\phi$ into which $\mathcal{H}$ injects, as well as a unital representation $\pi_\phi: \mathcal{A}_B \to B(\mathcal{H}_\phi)$ given by
\begin{equation}
\pi_\phi(b): a \otimes h \mapsto ba \otimes h.
\end{equation}
We have abused notation somewhat, as $a \otimes h$ now represents an equivalence class. A Herglotz representable free function may then be defined by 
\begin{equation}\label{eq:repconstruction}
g_\phi(X) = \tensor{R_\phi}{1_n}\left[\left(I + \tensor{\pi_\phi(u)}{I_n} \cdot \tensor{\pi_\phi}{1_n}(X)\right)\left(I - \tensor{\pi_\phi(u)}{I_n} \cdot \tensor{\pi_\phi}{1_n}(X)\right)^{-1}\right],
\end{equation}
where $R_\phi: B(\mathcal{H}_\phi) \to B(\mathcal{H})$ is the restriction map.

  \subsection{Regular Herglotz representations as the image of operator valued states}

We denote
$$\widehat{\chi}_B^\mathcal{H} = \repregherglotzclass{B}{B(\mathcal{H})}.$$
Bridging the gap between representations (\ref{eq:herglotzrep}) and (\ref{eq:repconstruction}) shows that the set of representable regular Herglotz functions is closed.

\begin{proposition} \label{closedresult}
If $B$ and $\mathcal{H}$ are fixed, then the space $\widehat{\chi}_B^\mathcal{H}$ of representable {\notred regular} Herglotz free functions is compact, and therefore closed, in the pointwise weak operator topology.
\end{proposition}
\begin{proof}
Suppose $f \in \widehat{\chi}_B^\mathcal{H}$ is written as in (\ref{eq:herglotzrep}). Then as $\alpha: B \to B(\widehat{\mathcal{H}})$ is a unital representation and $U \in \mathcal{U}(\widehat{\mathcal{H}})$, the universal property of $\overline{\mathcal{A}_B}$ produces a unital representation $\beta: \overline{\mathcal{A}_B} \to B(\widehat{\mathcal{H}})$ which extends $\alpha$ and satisfies $\beta(u) = U$. Now, $R \circ \beta: \overline{\mathcal{A}_B} \to B(\mathcal{H})$ is a linear, completely positive, unital map on a $C^*$-algebra, so its restriction $\phi$ to $\mathcal{A}_B$ is an operator-valued state on $\mathcal{A}_B$, and $g_\phi$ can be defined as in (\ref{eq:repconstruction}). Although $f$ and $g_\phi$ pass through distinct Hilbert spaces $\widehat{\mathcal{H}}$ and $\mathcal{H}_\phi$, the two functions are equal: for $b_1, \ldots, b_n \in B$, 
\begin{align}
R_\phi[\pi_\phi(u)\pi_\phi(b_1) &\pi_\phi(u) \pi_\phi(b_2) \cdots \pi_\phi(u) \pi_\phi(b_n)] \notag \\ &= R_\phi[\pi_\phi(ub_1ub_2\cdots ub_n)] \notag \\ &= 
\phi(u b_1 u b_2 \cdots u b_n) \notag \\ &= R[\beta(u b_1 u b_2 \cdots u b_n)] \notag \\ &= R[ U \alpha(b_1) U \alpha(b_2) \cdots U \alpha(b_n)],
\end{align}
so the geometric series expansions of (\ref{eq:herglotzrep}) and (\ref{eq:repconstruction}) agree. It follows that the association $\phi \mapsto g_\phi$ is a surjection from $\Phi_{\mathcal{A}_B}^{\mathcal{H}}$ onto $\widehat{\chi}_{B}^\mathcal{H}$. It now suffices to show that $\phi \mapsto g_\phi$ is continuous in the pointwise weak topology, as $\Phi_{\mathcal{A}_B}^{\mathcal{H}}$ is compact.

Fix $\varepsilon > 0$, $\vec{h}, \vec{j} \in \mathcal{H} \otimes \mathbb{C}^n$, and $X \in \textrm{Ball}(B) \cap (B \otimes M_n(\C))$. Because $R_\phi$ is a linear contraction and $g_\phi$ may be expanded in a geometric series, it follows that for some $k$,
\begin{equation}
\left|\left| g_\phi(X) - \tensor{R_\phi}{1_n}\left[I + 2 \sum\limits_{j = 1}^k \left(\tensor{\pi_\phi(u)}{I_n} \cdot \tensor{\pi_\phi}{1_n}(X)\right)^j \right] \right| \right| < \varepsilon.
\end{equation}
That is, $g_\phi(X)$ is approximated by a matrix whose entries are sums of terms of the form $\phi(ub_1ub_2\cdots ub_p)$ of bounded word length, where $b_1, \ldots, b_p$ are chosen from the entries of $X$. If $\psi \in \Phi_{\mathcal{A}_B}^\mathcal{H}$ is chosen such that for each of these terms,
\begin{equation}
\left| \langle \phi(ub_1ub_2\cdots ub_p)h_k, j_l \rangle - \langle \psi(ub_1ub_2 \cdots ub_p)h_k, j_l \rangle \right| < \varepsilon,
\end{equation}
then $\left| \langle g_\phi(X) \vec{h}, \vec{j} \rangle - \langle g_\psi(X) \vec{h}, \vec{j} \rangle  \right|$ is bounded by a finite multiple of $\varepsilon$ which only depends on $X$. Continuity and compactness follow.
\end{proof}

{\notred A representable regular Herglotz function is defined in reference to an auxiliary Hilbert space $\widehat{\mathcal{H}}$ that includes $\mathcal{H}$. We note that while $\widehat{\chi}_B^\mathcal{H}$ is closed, the choice of $\widehat{\mathcal{H}}$ may vary with each function.} Combining the unitary $U$ and representation $\alpha$ of (\ref{eq:herglotzrep}) into one representation $\pi_\phi$ allows us to view the Herglotz representation (\ref{eq:repconstruction}) in terms of a single parameter, $\phi$. However, the choice of $\phi$ is not unique.

\begin{proposition} \label{noninjectivity}
If the unital $C^*$-algebra $B$ is not one-dimensional, then the surjection $\phi \in \Phi_{\mathcal{A}_B}^{\mathcal{H}} \mapsto g_\phi \in \widehat{\chi}_B^\mathcal{H}$ fails to be injective.
\end{proposition}
\begin{proof}
Expanding the geometric series of (\ref{eq:herglotzrep}) for upper triangular matrices shows that $g_\phi$ determines $\phi(ub_1ub_2\cdots ub_n)$ for $b_1, \ldots, b_n \in B$, but such terms and their adjoints do not span all of $\mathcal{A}_B$. Suppose $B \not= \mathbb{C}$ is a unital $C^*$-algebra, so that the GNS construction produces at least two distinct states $\phi_1, \phi_2$ on $B$, which we may consider as linear maps from $B$ to $B(\mathcal{H})$ taking scalar values. Let $\gamma: \overline{\mathcal{A}_B} \to B \otimes C(\mathbb{S}^1) = C(\mathbb{S}^1, B)$ be the unique unital $*$-homomorphism with $\gamma(b) = b \otimes 1$ and $\gamma(u) = 1 \otimes \chi$, where $\chi(z) = z$ is the identity character on $\mathbb{S}^1$. \red Here $\mathbb S^1$ is the unit circle equipped with normalized arc length measure. \black Then there are two operator-valued states $\psi_1, \psi_2 \in \Phi_{\mathcal{A}_B}^\mathcal{H}$ defined by
\begin{equation} \psi_j(w) = \int_{\mathbb{S}^1} \phi_j(\gamma(w)).\end{equation}
 Note in particular that complete positivity and complete boundedness of $\psi_j$ are automatic, as $\phi_j$ is scalar-valued and $\psi_j$ is restricted from $\overline{\mathcal{A}_B}$.
Both $\psi_1$ and $\psi_2$ annihilate terms $u b_1 u b_2 \cdots u b_n$, and therefore the geometric series expansion of (\ref{eq:repconstruction}) shows that $g_\phi$ and $g_\psi$ are both the constant function $I$.
\end{proof}

\subsection{Regular Herglotz functions arising from noncommutative conditional expectations}

Suppose $\psi: \mathcal{A}_B \to B$ is a conditional expectation, so $\psi(b_1 m b_2) = b_1 \psi(m) b_2$ for $m \in \mathcal{A}_B$ and $b_1, b_2 \in B$; we denote this by $\psi \in E_{\mathcal{A}_B}^B$. Supposing $B$ has an injective representation into $B(\mathcal{H})$ (that is, $B$ is concrete), we may abuse notation somewhat and view $E_{\mathcal{A}_B}^B \subset \Phi_{\mathcal{A}_B}^\mathcal{H}$, with the intention to consider the Herglotz functions $g_\phi$ represented by these conditional expectations.

In terms of \emph{free probability}, one should think of the operator valued states which are also noncommutative conditional expectations  as distributions, as in \cite{will13}.
The ring $\mathcal{A}_B$  is essentially $B$ extended by a free unitary, that is, we have adjoined a unitary with no relations to $B.$
{\notred
In \cite{voi95} and \cite{voi00}, Voiculescu instead considered the ring resulting from the addition {\notred of} a free self-adjoint,
which is denoted by $B\langle X \rangle.$
Voiculescu and many subsequent authors have studied \dfn{(real) distributions}, maps
$\phi: B\langle X \rangle \rightarrow B$
which are noncommutative conditional expectations. Voiculescu introduced analogues of various classical integral transforms from probability and measure theory as special ``fully matricial maps", which are now understood as instances of free noncommutative functions.}
Classically, it was known that measures on the real line are in bijection with a certain class of functions which map the upper half plane to itself and satisfy good asymptotics at infinity via a Cauchy transform $\int_{\mathbb{R}} \frac{1}{t-z}d\mu(t)$ \cite{nev22}. Williams proved such a bijection exists in free probability when we restrict to the bounded case \cite{will13}.

We take the view that the subset of our so-called ``operator-valued states'' on $\mathcal{A}_B$ which are noncommutative conditional expectations corresponds to a \dfn{unitary distribution.}
That is, \red at least qualitatively, \black our work in this manuscript is to \cite{will13, 2015williams} as moment theory on the circle is to moment theory on the real line. \red We note that this analogy is witnessed by the relationship in the work of Agler \cite{ag90, aty13} between transfer function theoretic representations of Herglotz and Nevanlinna functions in two complex variables. Presumably, a similar relationship may exist here. \black

\begin{proposition}\label{prop:condclosure}
Let $B$ be a unital $C^*$-subalgebra of $B(\mathcal{H})$. The association $\phi \in E_{\mathcal{A}_B}^B \to g_\phi \in \widehat{\chi}_{B}^\mathcal{H}$ is noninjective if $B $ is not one-dimensional. Furthermore, if $B$ is weakly closed, then the set of Herglotz functions represented by conditional expectations $\phi \in E_{\mathcal{A}_B}^B$ is compact in the pointwise weak operator topology. 
\end{proposition}
\begin{proof}
{\bf Closedness.}
If $B$ is weakly closed, $E_{\mathcal{A}_B}^B \subset \Phi_{\mathcal{A}_B}^\mathcal{H}$ is closed (and therefore compact) in the pointwise weak operator topology. The collection of Herglotz functions $g_\phi$ represented by conditional expectations is a continuous image of $E_{\mathcal{A}_B}^B$ and therefore compact.

{\bf Noninjectivity in the generic case: $Aut(B) \not= \{1\}$.}

Given an automorphism $\psi$ of $B$, let $B \rtimes_\psi \mathbb{Z}$ denote the $C^*$-algebra crossed product. \blue Recall that $B \rtimes_\psi \mathbb{Z}$ is generated by $B$ and an additional unitary $\delta$ which implements the automorphism $\psi$ through conjugation. That is, $\delta$ formally satisfies the relation 
\[
 \delta b = \psi(b) \delta 
\]
for $b \in B$. For a more in-depth discussion, see \cite{dwilliams}.
\black There is a positive, unital, linear map $\alpha_\psi: B \rtimes_\psi \mathbb{Z} \to B$ defined by $\alpha_\psi( \sum b_k \delta^k) = b_0$. Precomposing $\alpha_\psi$ with a $*$-homomorphism from (the $C^*$-completion of) $\mathcal{A}_B$ to $B \rtimes_\psi \mathbb{Z}$ which fixes $B$ and sends $u$ to $\delta$ produces a conditional expectation $\widetilde{\psi} \in E_{\mathcal{A}_B}^B$. Regardless of the choice of $\psi$, $g_{\widetilde{\psi}}$ is the constant function $I$, as the terms $\widetilde{\psi}(ub_1ub_2\cdots ub_n)$ all vanish. However, $\widetilde{\psi}(ubu^*) = \psi(b)$, so if there are two distinct automorphisms $\psi_1$ and $\psi_2$ of $B$, then $\widetilde{\psi_1}$ and $\widetilde{\psi_2}$ are distinct and represent the same Herglotz function. This case applies to all noncommutative unital $C^*$-algebras, as a noncentral self-adjoint $b \in B$ produces at least one noncentral unitary $v = e^{itb}$, and conjugation by $v$ is a non-identity automorphism. On the other hand, for the commutative case $B = C(X)$, $B$ has nontrivial automorphism group if and only if $B$ is not rigid \cite{dow88}.

{\bf Noninjectivity in the singular case: $B = C(X)$ for rigid $X$.}

Let $\psi: B \to B$ be an endomorphism corresponding to a constant self-map of $X$. 
\red That is, noting that $B$ is isomorphic to $C(X)$, pick some $x \in X$ and define $\psi(b) = b(x)\cdot 1$ where $1$ is the constant function $1$. \black We define a representation of $B$
into $$\tensor{B}{B(\ell^2(\mathbb{Z}))} \subseteq
B(\mathcal{H} \overline{\otimes} \ell^2(\mathbb{Z}))$$
by 
$$\alpha(b) = \sum^{\infty}_{i=0} \tensor{\psi^{\chi_{i \neq 0}}(b)} { e_ie_i^*},$$ where $\{e_i\}_{i\in \mathbb{Z}}$ is the natural basis for $\ell^2(\mathbb{Z}),$
$\chi$ is a truth indicator function,
and $\psi^k$ represents $\psi$ iterated $k$ times.

We let $U$ be the identity tensored with the shift operator. 
Finally, define $R(A) = \tensor{I}{e_0^*}A\tensor{I}{e_0}.$
Consider
the associated operator-valued state $\widetilde{\psi}$ associated to $\alpha_\psi$, $U$, and $R.$ Since $R$ is a conditional expectation, so is $\widetilde{\psi}.$
Furthermore,
$\widetilde{\psi}(u^*bu) = \psi(b)$, and
$\widetilde{\psi}(ub_1ub_2\cdots ub_n)$ all vanish.
If $X$ is not a singleton, then repeating this construction for two different endomorphisms produces distinct representations for the function $h \equiv I.$
\end{proof}

It is unclear to the authors whether or not  $C(X)$ can be weakly closed when $X$ is a rigid space; this would say that $X$ is hyperstonean \cite{tak02}  in addition to being rigid. We should also emphasize that the noninjectivity established in Proposition \ref{prop:condclosure} implies that, in general,
free Herglotz functions arising from noncommutative conditional expectations do not determine unique noncommutative conditional expectations, which \emph{suggests that the correspondence between free function theory and free probability may be imperfect.}
{\notred Furthermore, we note that similar phenomena were noticed for the case of certain matrix-valued noncommutative Cauchy transforms of unbounded operators in} \cite[Corollary 5.2]{2015williams}.

If $\phi \in E_{\mathcal{A}_B}^B$, then evaluating $g_\phi$ at certain nilpotent matrices in $\textrm{Ball}(B)$ produces a compatibility condition that $g_\phi$ must satisfy. If $b_1, \ldots, b_n \in B$ are of small norm, then consider the application of $g_\phi$ to the nilpotent
\begin{equation}\label{magicnilpotent} 
X = \left( \begin{array}{cccccccc} 0 & b_1 \\ & 0 & b_2 \\ & & 0 & \ddots \\ & & & 0 & b_n \\ & & & & 0\end{array} \right),
\end{equation}
using the difference-differential calculus in Chapter 3 of \cite{vvw12}. Because $X$ has entries adjacent to the main diagonal, successive powers of $X$ have entries in different off-diagonals, so the geometric expansion of $g_\phi(X)$ from \eqref{eq:repconstruction} gives
\begin{equation}\label{eq:bigolmatrix}
g_\phi(X) = I_{n+1} + 2\bpm  0 & \phi(ub_1) & \phi(ub_1 ub_2)&  \cdots & \phi(ub_1 ub_2\cdots u b_n) \\ & 0 & \phi(ub_2)&  \cdots & \phi(ub_2\cdots ub_n) \\ & &\ddots \\ & & & 0 & \phi(ub_n)\\
& & & & 0 \epm.
\end{equation}
Every entry of the final column displayed in \eqref{eq:bigolmatrix} evaluates $\phi$ at a word ending in $b_n$. Because $\phi$ is a conditional expectation, if $X \in M_{n+1}(B)$ as in (\ref{magicnilpotent}) and $b \in B$ are both in the unit ball, then
\begin{equation}\label{eq:nilpotentstuff}
g_\phi(X (I_n \oplus b) ) - I_{n+1} = (g_\phi(X) - I_{n+1})(I_n \oplus b)
\end{equation}
follows. A partial converse also holds, as follows.

\begin{theorem} \label{condcond}
Suppose $B$ is a weakly closed, unital $*$-subalgebra of $B(\mathcal{H})$ and $g_\phi: \mathrm{Ball}(B) \to B(\mathcal{H})$ is a representable regular Herglotz function which satisfies (\ref{eq:nilpotentstuff}) for all $b \in B$ and $X \in M_{n+1}(B)$ as in (\ref{magicnilpotent}) with $||b|| < 1$ and $||X|| < 1$. Further, assume that the spectrum of $\pi_\phi(u)$ is a proper subset of the unit circle. Then $\phi$ is a conditional expectation.
\end{theorem}
\begin{proof}

The assumptions on $g_\phi$ imply that for any $b_1, \ldots, b_n, b_{n+1} \in B$, $\phi(ub_1\cdots ub_n b_{n+1}) = \phi(ub_1\cdots ub_n)b_{n+1}$, where we note that $\phi$ is linear and the $b_i$ may be scaled. Because the spectrum of $\pi_\phi(u)$ is a proper subset of the unit circle, $\pi_\phi(u)^*$ is in the closed algebra generated by $\pi_\phi(u)$, and $\phi(w b) = \phi(w) b$ holds for all $b \in B$ and $w \in \mathcal{A}_B$. Now, $\phi$ is positive and therefore preserves adjoints, so $\phi(b^* w^*) = b^* \phi(w^*)$. Together, these properties show that $\phi$ itself is a conditional expection. 
\end{proof}

We caution that our proof of Theorem \ref{condcond} relies heavily on our assumption that
the spectrum of $\pi_\phi(u)$ is a proper subset of the unit circle. Up to some M\"obius \blue transformations, \black that assumption corresponds to the compactly supported case from Williams \cite{will13}. \emph{A general identification of the set of free  functions which should arise from free probability remains unclear.}

\section{The construction of the Herglotz representation}


\subsection{The Agler model theory for noncommutative functions of Ball, Marx and Vinnikov}

The Cayley transform between the disk and the right half plane given by
	$$z = \frac{x + 1}{x - 1}, \hspace{.1in} x = \frac{1 - z}{1 + z}$$
extends to provide a one-to-one correspondence between free Herglotz functions and \dfn{free Schur functions}, that is, free functions $f: \BallB{B} \to \BallB{\BH}$.

Let $\delta: \MU{B_1} \rightarrow  \MU{B_2}$ {\notred be a free function.} We define the set of 
\dfn{$\delta$-contractions} to be
	$$G_\delta = \{X \in \MU{B_1} | \hspace{2pt} \|\delta(X)\|<1\}.$$

We recall the definition of a model, following the ``linear forms" approach from \cite{pastd14}.
\begin{definition}
Let $f: G_\delta \to \BallB{\BH}$. \red Let $S \subset G_\delta$ be a set closed under direct sums. \black
A \dfn{model} for $f$ on $S$ {\notred consists of a Hilbert space $\hilbert$ and a graded function $u$} such that 
	$$T - f(X)\ad T f(X) = u(X)\ad [T - \tensor{I_\hilbert}{\delta(X)\ad T \delta(X)}] u(X),$$
where $X \in {\notred S \cap M_n(B)}$ and $T \in M_n.$
\end{definition}

We also recall the following theorem on models which was proven by Agler, {\mc}Carthy \cite{agmc_gh} in the special case of polynomially convex sets and by Ball, Marx and Vinnikov \cite[Corollary 3.2]{BMV16} in general.
\begin{theorem}[Ball, Marx and Vinnikov \cite{BMV16}]\label{ncschurmodel} 
Let $B_1, B_2$ be $C^*$-algebras.
Let $\delta$ be a free function.
Let $f:  G_\delta \to \BallB{\BH}$ be a free function. Then $f$ has a model \red on $G_\delta$ \black.
\end{theorem}


\subsection{Lurking isometry argument construction of Herglotz representation on a polynomially
convex set for Herglotz functions with a model}

\begin{lemma}\label{blockherg}
Suppose $\delta(0)=0.$ Let $S$ be a \red set contained in $G_\delta$ closed under direct sums and containing $0.$ \black
Suppose $h: S \rightarrow \RHPB{\BH}$ has
$h(0)=I$.
If the Cayley transform of $h$ has a model,
then there exists:
\begin{enumerate}
\item A $C^*$-algebra $M$ unitally containing $B,$
\item A completely positive linear map $R: M \rightarrow \BH,$
\item A unitary $U\in M,$
\end{enumerate}
such that
\beq \label{herglotz} h(X) = \tensor{R}{\text{id}_n} \left[\left( I - \tensor{U}{I_n}\delta(X)\right)\left(I + \tensor{U}{I_n}\delta(X)\right)^{-1}\right],\eeq 
which, additionally, allows $h$ to be continued to a free function
$h: G_{\delta} \rightarrow \RHPB{\BH}.$
Moreover, if $h$ has a representation, then $h$ has a model.
\end{lemma}

\begin{proof}

Following Agler \cite{ag90}, we use a standard lurking isometry argument.

	$\Rightarrow:$ Suppose that $h:  S \to \RHPB{\BH}$ is a free Herglotz function with $h(0) = I$. Define a free Schur function $f: S \to  \BallB{\BH}$ by the Cayley transform	
	$$f(X) = (h(X) - I)(h(X) + I)\inv.$$
		By Theorem \ref{ncschurmodel}, $f$ has a model given by 
	\[
		T - f(X)\ad T f(X) = u(X)\ad [T - \tensor{I_\hilbert}{\delta(X)\ad T \delta(X)}] u(X),
	\] 
	which after substitution becomes
	\[
		T - [(h(X) - I)(h(X) + I)\inv ]\ad T(h(X) - I)(h(X) + I)\inv
		= u(X)\ad [T - \tensor{I_\hilbert}{\delta(X)\ad T \delta(X)}] u(X).
	\]
	Multiplying on the left and right by $(h(Y) + I)\ad$ and $(h(X) + I)$ respectively, we get
	\begin{align*}
		(h(X) + I)\ad T &(h(X) + I) - (h(X) - I)\ad T (h(X) - I) \\  
		&= \left[u(X)(h(X) - I)\right]\ad \left[T - \tensor{I_\hilbert}{\delta(X)\ad T \delta(X)}\right] \left[u(X) (h(X) + I)\right].
	\end{align*}
Let $v(X) = u(X)(h(X) + I)$. Upon substitution and rearrangement, the above equation becomes
\begin{align}
	(h(X) + I)\ad T(h(X) + I) + &v(X)\ad \tensor{I_\hilbert}{\delta(X)\ad T \delta(X)} v(X) \notag \\
	&= (h(X) - I)\ad T(h(X) -I) + v(X)\ad \tensor{I_\hilbert}{T} v(X). \label{sorted}
\end{align}
This expression can be rewritten as a Gramian by
\beq \label{gramian}
	\vectwo{h(X) + I}{\tensor{I_\hilbert}{\delta(X)}v(X)}\ad T \vectwo{h(X) + I}{\tensor{I_\hilbert}{X}v(X)} = \vectwo{h(X) - I}{v(X)}\ad T \vectwo{h(X) - I}{v(X)}.  
\eeq
 
Define 
	$$\phi(X) = \vectwo{h(X) + I}{\tensor{I_\hilbert}{\delta(X)}v(X)}, \text{ and } \theta(X) = \vectwo{h(X) - I}{v(X)}$$
and note that by a standard lurking isometry theorem (see \cite{pastd14}),  there exists a unitary $L$ such that
	$$\tensor{L}{I}\phi(X) = \theta(X).$$
$L$ can be written in block form as
\[
L = \begin{bmatrix} A & B \\ C & D \end{bmatrix}  \\
\]
so that 
\[
\bbm \tensor{A}{I_n} & \tensor{B}{I_n} \\ \tensor{C}{I_n} & \tensor{D}{I_n} \ebm \vectwo{h(X) + I}{\tensor{I_\hilbert}{\delta(X)} v(X)} = \vectwo{h(X) - I}{v(X)}.
\]
Multiplying through gives the relations
\begin{align*}
&\tensor{A}{I_n} (h(X) + I) + \tensor{B}{I_n} \tensor{I_\hilbert}{\delta(X)}v(X) =  (h(X) - I), \\
&\tensor{C}{I_n} (h(X) + I) + \tensor{D}{I_n} \tensor{I_\hilbert}{\delta(X)}v(X) =  v(X).
\end{align*}
Arranging this system to facilitate elimination of $h$, we get
\begin{align*}
\tensor{I + A}{I_n} &h(X) + \tensor{B}{I_n}\tensor{I_\hilbert}{\delta(X)} v(X) = \tensor{I - A}{I_n} \\
\tensor{C}{I_n}&h(X) + \left(\tensor{D}{I_n}\tensor{I_\hilbert}{\delta(X)} - I\right)v(X) = \tensor{-C}{I_n}.
\end{align*}
Multiplying through the first equation by $\tensor{(I + A)\inv}{I}$ and then $\tensor{C}{I}$ gives the system
\begin{align*}
\tensor{C}{I_n} &h(X) + \tensor{C(I + A)\inv B}{I_n}\tensor{I_\hilbert}{\delta(X)}v(X) = \tensor{C(I+A)\inv (I - A)}{I_n} \\
\tensor{C}{I_n}&h(X) + \left(\tensor{D}{I_n}\tensor{I_\hilbert}{\delta(X)} - I\right)v(X) = \tensor{-C}{I_n}.
\end{align*}
Subtracting the equations to eliminate $h(x)$ gives
\beq \label{prelemma}
\left[I - \left(\tensor{D}{I_n} - \tensor{C(I+A)\inv B}{I_n}\right)\tensor{I_\hilbert}{\delta(X)}\right]v(X) = \tensor{C(I + A)\inv (I - A) + I}{I_n}.
\eeq

We now need the following lemma.

\begin{lemma}[\cite{ag90}] \label{herglotzunitary}
	Let $\hilbert_1$ and $\hilbert_2$ be Hilbert spaces, let $L$ be a bounded operator on $\hilbert_1 \oplus \hilbert_2$ with block form
	\[
		L = \bbm A & B \\ C & D \ebm,
	\]
	assume that $1 \notin \sigma(A)$, and let $U = D - C(1+A)\inv B$. If $L$ is an isometry, then so is $U$, and if $L$ is unitary, then so is $U$. 
\end{lemma}

Since $h(0) = I$ implies that $A = 0$ and thus that $1\notin \sigma(A)$, we apply Lemma \ref{herglotzunitary} to \eqref{prelemma}. Consequently, the operator $U = D - C(1+A)\inv B$ is unitary, which allows the simplification of Equation \eqref{prelemma} to
\[
\left(I - \tensor{U}{I_n}\tensor{I_\hilbert}{\delta(X)}\right)v(X) = \tensor{2C(I+A)\inv}{I_n}.
\]
Let $V = C(I+A)\inv$. Then we get the following expression for $v(X)$, where the inverse operator exists as $X$ is a strict contraction:
\beq \label{sub1}
v(X) = \left(I - \tensor{U}{I_n}\tensor{I_\hilbert}{\delta(X)}\right)\inv \tensor{2V}{I_n}.
\eeq
Another useful form of this equation is
\beq \label{sub2}
v(X) = \tensor{U}{I_n}\tensor{I_\hilbert}{\delta(X)} v(X) + \tensor{2V}{I_n}.
\eeq

By evaluating the model at 
$$T = \bbm 0 & 1 \\ 1 &0 \ebm, \hat{X} = \bbm 0 & 0 \\ 0 & X \ebm,$$
one obtains that \eqref{sorted} can be rewritten as 
\beq \label{simplesorted}
2(h(0)\ad + h(X)) =  v(0)\ad v(X) - \left(\tensor{I_\hilbert}{\delta(0)}v(0)\right)\ad\tensor{I_\hilbert}{\delta(X)}v(X).
\eeq
Substituting \eqref{sub2} into \eqref{simplesorted}, we have 
\beq \label{almostrep}
2(h(0)\ad + h(X)) = \left(\tensor{U}{I_n}\tensor{I_\hilbert}{\delta(0)}v(0)\right)\ad\tensor{2V}{I_n} + \tensor{2V\ad U}{I_n}\tensor{I_\hilbert}{\delta(X)}v(X) + \tensor{2V\ad}{I_n}\tensor{2V}{I_n}.
\eeq

As $h(0) = I$, evaluating $X = 0$ produces
\[
4 = \tensor{2V\ad}{I_n} \tensor{2V}{I_n},
\]
and therefore $V$ is an isometry.
Now set $Y = 0$ in \eqref{almostrep} and divide by $2$, which gives
\[
	h(X) = \tensor{V\ad}{I}\tensor{U}{I}\tensor{I_\hilbert}{\delta(X)}v(X) + I.
\]
By substitution of Equation \eqref{sub1} for $v(X)$,
\begin{align*}
h(X) &= \tensor{V\ad}{I_n}\tensor{U}{I_n}\tensor{I_\hilbert}{\delta(X)}\left(I - \tensor{U}{I_n}\tensor{I_\hilbert}{X}\right)\inv \tensor{2V}{I_n} + I \\
&= \tensor{V\ad}{I_n}\left[2\tensor{U}{I_n}\tensor{I_\hilbert}{\delta(X)}\left(I - \tensor{U}{I_n}\tensor{I_\hilbert}{\delta(X)}\right)\inv + I\right]\tensor{V}{I_n} \\
&= \tensor{V\ad}{I_n}\left[\left(I + \tensor{U}{I_n}\tensor{I_\hilbert}{\delta(X)}\right)\left(I - \tensor{U}{I_n}\tensor{I_\hilbert}{\delta(X)}\right)\inv\right]\tensor{V}{I_n}.
\end{align*}

\black
So, there exists
a Hilbert space $\mathcal{K},$ an isometry $V: \mathcal{K} \rightarrow \hilbert$, and a unitary $U$
such that
$$h(X)= \tensor{V\ad}{I_n}\left[\left(I + \tensor{U}{I_n}\tensor{I_\hilbert}{\delta(X)}\right)\left(I - \tensor{U}{I_n}\tensor{I_\hilbert}{\delta(X)}\right)\inv\right]\tensor{V}{I_n}.$$
Finally, we let $M$ be the $C^*$-algebra generated by $U$ and $B$ inside $\BH \otimes B,$
and define $R(a) = V^*aV$, completing the argument. (Notably, we are viewing $B$ as a unital subalgebra
of $M$ via the map $b \rightarrow \tensor{I_{\hilbert}}{b}$.)

$(\Rightarrow)$ The converse follows from the observation that the calculation above is reversible.

\end{proof}

Now, the result immediately implies Theorem \ref{exacthasrep}.

For a nonconstant free Herglotz function $h$, $h(0) = S + iT$ where $S \geq 0$, and

	$$\hat{h}(X) = \tensor{S^{-1/2}}{I_n} \left(-i\tensor{T}{I_n} + h(X)\right)\tensor{S^{-1/2}}{I_n}$$
is a free Herglotz function with $\hat{h}(0) = I$. \red Here, when $\ker S \neq \{0\}$, we define $S^{-1/2}$ to be $0$ on $\ker S$ and $(P_{\ker S^\perp} S |_{\ker S^\perp})^{-1/2}$ on $\ker S^\perp$. \black (If $S$ has a kernel or $S^{-1/2}$ is unbounded,
an additional normal families argument is required. We leave this to the interested reader.)
This normalization allows the following corollary.

\begin{corollary} \label{finalcor}
Let $h: \BallB{B_1} \rightarrow \RHPB{B_2}$.
Then there exists:
\begin{enumerate}
	\item A $C^*$-algebra $M$ unitally containing $B,$
	\item	A completely positive linear (not necessarily unital) map $R: M \rightarrow B_2,$ 
	\item A unitary $U \in M,$
	\item A bounded self-adjoint operator $T,$
\end{enumerate}
such that 
\beq \label{herglotzfinal} 
h(X) =\tensor{ iT}{I_n} + \tensor{R}{1_n}
\left[\left( I + \tensor{U}{I_n}X\right)\left(I - \tensor{U}{I_n}X\right)^{-1}\right].
\eeq
\end{corollary}

\section{Acknowledgments}
The authors would like to thank Michael Anshelevich and John Williams for introducing them to this class of problems. Additionally, we should mention that an earlier circulated version of this paper had several results which were conditional on exactness or the existence of Agler model theory. However, with the advent of the Ball-Marx-Vinnikov framework \cite{BMV16}, these difficulties were removed.

The authors would also like to recognize the referee's generous contribution of time, in particular providing an extensive historical context and relevant references for the present work.

\bibliography{references}
\bibliographystyle{plain}

\end{document}